\documentclass[12pt]{amsart}
\usepackage{fullpage}
\usepackage{hyperref}
\usepackage[alphabetic]{amsrefs}

\usepackage{amsmath, amssymb, graphics, setspace}

\newcommand{\pEMD}{{\mathbb{EMD}}}
\newcommand{\cEMD}{{\mbox{EMD}_s}}

\newcommand{\bbM}{{\mathbb{M}}}
\newcommand{\bbD}{{\mathbb{D}}}
\newcommand{\bbP}{{\mathbb{P}}}
\newcommand{\bbR}{{\mathbb{R}}}
\newcommand{\bbN}{{\mathbb{N}}}

\newcommand{\V}{\mathbb{R}^n}

\newcommand{\mcE}{{\mathcal{E}}}
\newcommand{\mcJ}{{\mathcal{J}}}
\newcommand{\mcP}{{\mathcal{P}}}
\newcommand{\mcC}{{\mathcal{C}}}
\newcommand{\mcA}{{\mathcal{A}}}
\newcommand{\mcI}{{\mathcal{I}}}
\newcommand{\mcR}{{\mathcal{R}}}
\newcommand{\mcM}{{\mathcal{M}}}
\newcommand{\mcN}{{\mathcal{N}}}

\newcommand{\la}{\lambda}

\newtheorem{thm}{Theorem} 

\newtheorem{prop}[thm]{Proposition}
\newtheorem{conj}{Conjecture}

\theoremstyle{remark}

\numberwithin{equation}{section}

\begin{document}

\title{Expected value of the one-dimensional \\ Earth Mover's Distance}

\author[R.~Bourn]{Rebecca Bourn}
\address[Rebecca Bourn]{Department of Mathematical Sciences\\
         University of Wisconsin - Milwaukee\\
         3200 North Cramer Street \\
         Milwaukee, WI 53211}
\email[R.~Bourn]{bourn@uwm.edu}

\author[J.~F.~Willenbring]{Jeb F. Willenbring}
\address[Jeb F.~Willenbring.]{Department of Mathematical Sciences\\
         University of Wisconsin - Milwaukee\\
         3200 North Cramer Street \\
         Milwaukee, WI 53211}
\email[J.~F.~Willenbring]{jw@uwm.edu}

\date{\today}



\begin{abstract}
From a combinatorial point of view, we consider the Earth Mover's Distance (EMD) associated with a metric measure space.  The specific case considered is deceptively simple:  Let the finite set of integers $[n] = \{1,\cdots,n\}$ be regarded as a metric space by restricting the usual Euclidean distance on the real numbers.  The EMD is defined on ordered pairs of probability distributions on $[n]$.  We provide an easy method to compute a generating function encoding the values of EMD in its coefficients, which is related to the Segr{\'e} embedding from projective algebraic geometry.  As an application we use the generating function to compute the expected value of EMD in this one-dimensional case.  The EMD is then used in clustering analysis for a specific data set.
\end{abstract}

\maketitle

\section{Introduction}\label{sec:Introduction}

Fix a positive integer $n$.  We will denote the finite set of integers $\{1,\cdots, n\}$ by $[n]$.
By a \emph{probability measure} on $[n]$ we mean, as usual, a non-negative real-valued function $f$ on the set $[n]$ such that $f(1) + \cdots + f(n) = 1$.  By the \emph{probability simplex} on $[n]$ we mean the set of all probability measures on $[n]$, denoted $\mcP_n$.  We view $\mcP_n$ as embedded in $\bbR^n$.  Given $\mu, \nu \in \mcP_n$ define the set of joint distribution
\[
    \mcJ_{\mu \nu} = \left\{ J \in \bbR^{n \times n} :
        \begin{array}{l}
            J \mbox{ is a non-negative real number $n$ by $n$ matrix such that } \\
            \sum_{i=1}^n J_{ij} = \mu_j \mbox{ for all } j  \mbox{ and }
            \sum_{j=1}^n J_{ij} = \nu_i \mbox{ for all } i
        \end{array}
    \right\}.
\]
For results concerning the geometry of $\mcJ_{\mu\nu}$ see \cite{DK14} and \cite{PSU19} where they are referred to as
\emph{transportation polytopes} and \emph{discrete copulas} respectively.

The \emph{Earth Mover's Distance} is defined as
\[
    \pEMD(\mu, \nu) = \inf_{J \in \mcJ_{\mu \nu}} \sum_{i,j = 1}^n |i-j| J_{ij}.
\]
We remark that the set $\mcP_n \times \mcP_n$ is a compact subset of $\bbR^{2n}$ and so by continuity the infinum is actually a minimum value.  Also, the ``EMD'' is sometimes referred to by other names, for example,
in a two-dimensional setting it is called the \emph{image distance}.
More generally it is called the Wasserstein metric (see \cite{Mgw}).

We recall that the  set of all finite distributions, $\mcP_n$, embeds as a compact polyhedron on a hyperplane in $\bbR^n$ and inherits Lebesgue measure and has finite volume.  We normalize this measure so that the total mass of $\mcP_n$ is one.  We then obtain a probability measure on $\mcP_n$, which is uniform.  Similarly, $\mcP_n \times \mcP_n$ may be embedded in $\bbR^n \times \bbR^n$ and can be given the (uniform) product probability measure.

From its definition, the function $\pEMD$ is a metric on $\mcP_n$.  The subject of this paper concerns the expected value of $\pEMD$ with respect to the uniform probability measure.
In this light, we define a function $\mcM$ on ordered pairs of non-negative integers, $(p,q)$, as
\begin{equation} \label{eqn:mpq}
    \mcM_{p,q} = \frac{(p-1)\mcM_{p-1,q} + (q-1)\mcM_{p,q-1} + |p-q|}{p+q-1}
\end{equation}
with $\mcM_{p,q} = 0$ if either $p$ or $q$ is not positive.    Let $\mcM_n = \mcM_{n,n}$ for any non-negative integer $n$.

We will prove the following theorem in Section \ref{sec:Proofs of the Theorems}.

\begin{thm} \label{thm:main}
    Fix a positive integer $n$.  Let $\mcP_n \times \mcP_n$ be given the uniform probability measure defined by Lebesgue measure from the embedding into $\bbR^{2n}$.  The expected value of $\pEMD$ on $\mcP_n \times \mcP_n$ is $\mcM_n$.
\end{thm}

From a theoretical point of view, this paper concerns the expected value of $\pEMD$.  Additionally, we consider a discrete version of the EMD and compute the mean.  In turn, we discuss the relationship to cluster analysis.  Then, we end with a comparison of the theoretical results to grade distributions where we have noticed persistent clustering.

The above theorem is obtained as a limit of a discrete version of $\pEMD$ (denoted by $\cEMD$, for non-negative integer $s$) which is described using a generating function.
The generating function is a deformation of the Hilbert series of the
Segr{\'e} embedding.  Some standard tools from algebraic combinatorics show up in a new way in the proofs.

From a practical point of view, we will also consider a ``real world'' data set with a finite number of joint probability measures derived from letter grade distributions.  Specifically, we consider a network of grade distributions from the University of Wisconsin - Milwaukee campus, where two nodes are joined when the $\cEMD$ between them falls below a pre-specified distance threshold. Determining this threshold so that data features are revealed is a subject of research.   The expected value of the EMD in the uniformly random situation helps guide this choice.

The family of networks obtained by varying the distance threshold will be used for metric hierarchical clustering.  When the threshold is set so that the network is connected, the spectrum of the corresponding Laplacian matrix (see \cite{FC}) will be computed, as it also relates to clustering.  Our use of the EMD in this context should be viewed as an attempt at exploratory data analysis to identify the ``communities'' in this network rather than rigorous hypothesis testing.

\bigskip
\noindent {\bf Acknowledgments:}  The second author would like to thank Anthony Gamst for conversation concerning cluster analysis (and especially the reference \cite{K}), and Ryan (Skip) Garibaldi for pointing out the relationship between the Earth Mover's Distance and comparing grade distributions.

We would also like to thank the referees for making substantial improvements to the mathematical content and exposition of this article.

\bigskip

\section{Non-technical preliminaries}\label{sec:Non-tech Preliminaries }

In this section we consider some specific examples of finite distributions.  A motivating situation comes from grade distributions, which in the United States are often considered with five outcomes: A, B, C, D, and F.  The standard Grade Point Average (GPA) assigns 4.0 to A, 3.0 to B, 2.0 to C, 1.0 to D, and 0.0 to F.  The relative distances between these five grades is computed by the absolute value of the difference of point values.  That is, a B grade is three units away from an F, while one unit away from an A.

Suppose we are given three distributions in the five grade setting for classes with $s=30$ students.  That is,
\[
    \begin{array}{c|ccccc}
          &  A &  B &  C &  D &  F \\ \hline
        X &  0 & 19 &  8 &  2 &  1 \\
        Y & 12 &  2 &  5 & 11 &  0 \\
        Z &  2 & 20 &  2 &  3 &  3
    \end{array}
\]
To compare distribution X to distribution Y, one notices that if the 12 A grades in Y were moved down to B, 5 C grades moved up to B, 8 D grades moved up to C, and one D grade moved down to F, then the distributions would be identical.  The matrix
\[
\left[
\begin{array}{ccccc}
 0 & 0 & 0 & 0 & 0 \\
 12 & 2 & 5 & 0 & 0 \\
 0 & 0 & 0 & 8 & 0 \\
 0 & 0 & 0 & 2 & 0 \\
 0 & 0 & 0 & 1 & 0 \\
\end{array}
\right]
\]
encodes the ``conversion''.  That is, if the rows and columns correspond to the grades (A,B,C,D,F) then the entry in row i and column j records how many grades to move from position i in Y to position j in X.  The entries on the diagonal reflect no ``Earth'' movement, while the entries in the first sub and super diagonals reflect one unit of movement.  The row sums return the X distribution, while the column sums return the Y distribution.  In total, the value of $\cEMD$ is 26.

The Y and Z distributions compare as follows: move 10 B's up one unit to a grade of A, to reflect the fact that Y had 12 grades of A.   We move 5 grades down from B to C, and 3 grades from B to D.  This latter change is noted as a jump across two positions which will ``cost'' 2 units in EMD, and since there are 3 grades to move, this makes an overall contribution of 6.  Finally, 2 C grades are moved down to D, and the 3 F grades in the Z distribution would be moved to D in the distribution Y.

The joint distribution matrix for Y and Z  is
\[
\left[
\begin{array}{ccccc}
 2 & 10 & 0 & 0 & 0 \\
 0 & 2 & 0 & 0 & 0 \\
 0 & 5 & 0 & 0 & 0 \\
 0 & 3 & 2 & 3 & 3 \\
 0 & 0 & 0 & 0 & 0 \\
\end{array}
\right].
\]
Interestingly, the total EMD is again 26.

Finally, the X and Z distributions are compared.  The joint matrix is
\[
\left[
\begin{array}{ccccc}
 0 & 0 & 0 & 0 & 0 \\
 2 & 17 & 0 & 0 & 0 \\
 0 & 3 & 2 & 3 & 0 \\
 0 & 0 & 0 & 0 & 2 \\
 0 & 0 & 0 & 0 & 1 \\
\end{array}
\right].
\]
The total EMD is 10.

The three distributions above have the same GPA of 2.5.  We note that this is a feature, which we point out to give indication that the EMD will clearly distinguish between distributions even if the GPA is constant.

To help the reader gain some intuitive feel for the EMD we augment the three distributions by:
\[
    \begin{array}{c|ccccc}
        & A &  B & C &  D & F \\ \hline
U& 13& 13& 0& 0& 4 \\
V& 9& 1& 13& 2& 5 \\
W& 9& 7& 8& 6& 0
\end{array}
\]

As an exercise, one can compute the 36 pairwise distances between each of the six provided distributions.  We computed them with Mathematica as shown below:
\[
\begin{array}{c|cccccc}
 \text{EMD} & \text{U} & \text{V} & \text{W} & \text{X} & \text{Y} & \text{Z} \\ \hline
 \text{U} & 0 & 24 & 20 & 24 & 24 & 18 \\
 \text{V} & 24 & 0 & 12 & 26 & 16 & 22 \\
 \text{W} & 20 & 12 & 0 & 16 & 10 & 16 \\
 \text{X} & 24 & 26 & 16 & 0 & 26 & 10 \\
 \text{Y} & 24 & 16 & 10 & 26 & 0 & 26 \\
 \text{Z} & 18 & 22 & 16 & 10 & 26 & 0 \\
\end{array}
\]

One can sample distributions of five grades with 30 students in many distinct ways.   For each sampling method one can ask how the EMD is distributed.  The sampling method could be chosen to accurately simulate synthetic data to match previously observed samples from a particular subject at a particular institution.   Or, a prior distribution on grade distributions could be assumed, such as a discretization of the multivariate normal distribution.

Upon exploration of observed data one notices clear clustering of the distributions relative to the EMD.  Indeed, if some distributions are encountered more frequently than others in a particular model then clustering should be expected.   With this fact in mind one is led to question of sampling distributions at flat random.
That is, sampling independently with each distribution being equally likely.  The theoretical behavior of the uniform model can then be compared to observed data.  Clustering in the uniform model can be considered ``random'', while additional observed clustering in a specific data set is likely related to a causal feature.   Statistics describing clustering should be understood for the uniform distribution as it has maximal entropy.

For any given distribution, one seeks a theoretical understanding of any given descriptive statistic.   The present article restricts the focus to the mean of EMD.  Other statistics will be considered in future work.  Moreover, we focus on the uniform distribution only over the space of finite probability distributions.

Finally, we note that the results of this article imply that the mean discrete EMD on 30 student, five grade distributions is slightly larger than 26.  The maximum EMD is 120 reflecting the fact that the distance between all 30 students with A grades is 120 units away from the distribution with all 30 students with F grade.   Such large values of EMD are unlikely.  The distribution of actual grade data, as we expect, is skewed to the right (i.e. mean larger than median).

\section{Technical Preliminaries}\label{sec:Preliminaries}

In this section we  recall some basic notation from combinatorics and linear algebra that are used throughout the paper.

\subsection{Notation from linear algebra}\label{subsec:Notation}

Let  $\bbM_{n,m}$ be the vector space of real matrices with $n$ rows and $m$ columns.  Throughout, we assume that the field of scalars is the real numbers, $\mathbb{R}$.  For $i$ and $j$ with $1 \leq i \leq n$, $1 \leq j \leq m$ we let $e_{i,j}$ denote the $n$ by $m$ matrix with 1 in the $i$-th row, $j$-th column, and 0 elsewhere.  A matrix, $M \in \bbM_{n,m}$ is written as $M = (M_{i,j})$ where $M_{i,j}$ is the entry in the $i$-th row and $j$-th column.  So, $M = \sum M_{i,j} \, e_{i,j}$.  We assume standard notation for the algebra of matrices.  For example, the standard inner product of $X,Y \in \bbM_{n,m}$ is
\[
    \langle X, Y \rangle = \mbox{Trace}(X^T Y).
\]

In the case that $m=1$ we write $e_i = e_{i,1}$ for $1 \leq i \leq n$.  As usual, Let $\V$ denote the $n$-dimensional real vector space consisting of column vectors of length $n$.  The set $\{e_1, \cdots, e_n \}$ is a basis for $\V$.  For our purposes a very useful alternative basis is given by
\[
    \omega_j = \sum_{i=1}^j e_i
\]
where $1 \leq j \leq n$.    We call the set of $\omega_j$ the \emph{fundamental basis for $\V$}.  The terminology here comes from the the root system of type A in Lie theory (see \cite{GW}).

Let the orthogonal complement, $\omega_n^\perp$, to $\omega_n$ be denoted by
\[
    \V_0 = \left\{ v \in \V \mid \langle v, \omega_n \rangle = 0 \right\}.
\]
Column vectors in $\V_0$ have coordinates that sum to zero.  We let $\pi_0$ denote the orthogonal projection from $\V$ onto $\V_0$,
\[
    \begin{array}{cccc}
        \pi_0 : & \V & \rightarrow & \V_0 \\
                & v  & \mapsto     & \pi_0(v)
    \end{array}
\] defined by the formula
\[
    \pi_0(v) = v - \frac{\langle v, \omega_n \rangle}{n} \omega_n.
\]
Note that the image of $\pi_0$ is $\V_0$, and the kernel contains $\omega_n$.  For $1 \leq j \leq n-1$, let
$\tilde \omega_j = \pi_0(\omega_j)$.  Observe that $\tilde \omega_1, \cdots, \tilde \omega_{n-1}$ span $\V_0$, and by considering dimension, are a basis for $\V_0$.  We will call this set the \emph{fundamental basis for $\V_0$}.

The subspace $\V_0$ has another basis that is of importance to us:
\[
    \Pi = \{ \alpha_1, \cdots, \alpha_{n-1} \}
\] where $\alpha_j = e_j - e_{j+1}$.  We refer to $\Pi$ as the \emph{simple basis for $\V_0$}.  An essential point is that $\Pi$ is dual to the fundamental basis.  That is $\langle \alpha_i, \tilde \omega_j \rangle = \delta_{i,j}$ where
\[
    \delta_{i,j} = \left\{
        \begin{array}{ll}
            1 & \mbox{if $i=j$} \\
            0 & \mbox{otherwise}
        \end{array}
    \right.
\]

Let $\mcE: \V \rightarrow \bbR$ be defined as
\[
    \mcE(v) = |v_1| + |v_1 + v_2| + |v_1 + v_2 + v_3| + \cdots + |v_1 + \cdots + v_n|
\]
for $v = \sum v_j e_j \in \V$.

Observe that since $\V_0 \subset \V$, $\mcE$ is defined on $\V_0$ by restriction.

Let $v = \sum_{i=1}^{n-1} c_i \alpha_i \in \V_0$.  Then, $\mcE(v) = \sum_{i=1}^{n-1} |c_i|$.  This fact is easily seen since the fundamental basis is dual to the simple basis relative to the standard inner product, and
\[
    \langle v, \omega_j \rangle = v_1 + \cdots + v_j
\]
for $1 \leq j \leq n$.

In Section \ref{sec:Proofs of the Theorems} we will prove that

\begin{thm}\label{thm:L1} For all $\mu, \nu \in \mcP_n$,
\[
\pEMD(\mu,\nu) = \mcE(\mu-\nu).
\]
\end{thm}
This allows for a much more explicit combinatorial analysis of EMD.  For situations in which the metric space is not a subset of the real line, the analysis is more difficult.  Indeed, $\pEMD$ is often computed as an optimization problem that minimizes the cost under the constraints imposed by the marginal distribution.  Consequently, the computational complexity is the same as for linear programming.

\subsection{Compositions and related combinatorics}

Let $\bbN$ be the set of non-negative integers.
Given $s \in \bbN$, and a positive integer $n$, define:
\[
    \mcC(s,n) = \{ (a_1, a_2, \cdots, a_n) \in \bbN^n : a_1 + \cdots + a_n = s \}.
\]
An element of the set $\mcC(s,n)$ will be referred to as a \emph{composition of $s$ into $n$ parts}.  (We note that in some places of the literature these are refered to as \emph{weak compositions}, since we allow zero. However, the distinction is not needed for us.)

It is an elementary fact that there are $\binom{s+n-1}{n-1}$ compositions, and therefore for fixed $n$, the number of compositions of $s$ grows as a polynomial function of $s$ with degree $n-1$.  An essential fact for this paper is the asymptotic approximation
\[
    \binom{s+n-1}{n-1} \sim \frac{s^{n-1}}{(n-1)!}.
\]

As in the introduction, given compositions $\mu$ and $\nu$ of $s$ we let $\mcJ_{\mu \nu}$ denote the set of $n$ by $n$ matrices with row sums $\mu$ and column sums $\nu$.  There is a slight difference here in that we are not requiring $\mu$ and $\nu$ to be normalized to sum to one.  In the same light, EMD can be extended as a metric on $\mcC(s,n)$.

We fix an $n$ by $n$ matrix $C$ with $i$-th row and $j$-th column entry to be $|i-j|$.  That is,
\[
C = \left[
    \begin{array}{ccccc}
      0 & 1 & 2 & \cdots & n-1 \\
      1 & 0 & 1 & \cdots & n-2 \\
      2 & 1 & 0 & \cdots & n-3 \\
      \vdots & \vdots & \vdots & \ddots & \vdots \\
      n-1 & n-2 & n-3 & \cdots & 0
    \end{array}
  \right].
\]
So, for $\mu, \nu \in \mcC(s,n)$ and regarding the set $\mcJ_{\mu \nu}$ as non-negative integer matrices with prescribed row and column sums, we arrive at
\[
    \cEMD(\mu, \nu) = \min_{J \in \mcJ_{\mu \nu}} \langle J,C \rangle,
\]
which is a discrete version of EMD.  When we take $s \rightarrow \infty$ we recover the value referred to in the introduction.

The function EMD may be further generalized to the case where $C$ has $p$ rows and $q$ columns, with $i,j$ entry $|i-j|$.  In this case $\mu$ has $p$ components and $\nu$ has $q$ components (each a composition of $s$).  This generalization will be needed in an induction argument in Section \ref{sec:Proofs of the Theorems}.  However, applications need only consider the $p=q$ case.

A further generalization beyond the scope of this paper is to consider more general cost matrices than $C$.  This is equivalent to a variation of the metric.

\section{Generating Functions}\label{ec:Generating Functions}

In algebraic combinatorics it is often useful to record discrete data in a formal (multivariate) power series -- sometimes called a \emph{generating function}.  By ``formal'' we mean that the variables are indeterminates rather than numbers.
In fact, from this point of view one can consider formal power series that only converge at zero, yet encode combinatorial data in their coefficients.   Consequently, convergence is not an issue.
Nonetheless, our series are all geometric series expansions of rational functions, and so will be convergent if, say, all complex variables have modulus less than 1.

Starting from the viewpoint of algebraic combinatorics we define
\[
    H_n(z,t) := \sum_{s=0}^\infty \left( \sum_{(\mu,\nu) \in \mcC(s,n) \times \mcC(s,n)}
                z^{\tiny \cEMD(\mu,\nu)} \right) t^s,
\] where $t$ and $z$ are indeterminates.  We see that the coefficient of $t^s$ in $H_n(z,t)$ is a polynomial in $z$ whose coefficients record the distribution of the values of $\cEMD$.

It is useful to see the first few values of $H$, which we compute using Mathematica and Theorem \ref{thm:HS}:

\begin{eqnarray*}
H_1(z,t) & = & \frac{1}{1-t} \\[2mm]
H_2(z,t) & = & \frac{t z+1}{(1-t)^2 (1-t z)} \\[2mm]
H_3(z,t) & = & \frac{-t^3 z^4-t^2 (2 z+1) z^2+t (z+2) z+1}{(1-t)^3 (1-t z)^2 \left(1-t z^2\right)} \\
\end{eqnarray*}

As before, when considering $p$ by $q$ matrices we can analogously define $H_{p,q}(z,t)$.  We also extend the definition so that $H_{p,q} = 0$ if either of $p$ or $q$ is not positive.  We obtain a similar series, namely
\[
    H_{p,q}(z,t) := \sum_{s=0}^\infty \left( \sum_{(\mu,\nu) \in \mcC(s,p) \times \mcC(s,q)}
                z^{{\tiny \cEMD}(\mu,\nu)} \right) t^s.
\] Note that if $p<q$ we can regard $p$-tuples as $q$-tuples  by appending zeros, so the function $\cEMD$ is defined.

\begin{thm}\label{thm:HS}
For positive integers $p$ and $q$,
\[ H_{p,q}(z,t) = \frac{ H_{p-1,q}(z,t) + H_{p,q-1}(z,t) - H_{p-1,q-1}(z,t) }{1- z^{|p-q|} t}  \]
if $(p,q) \neq (1,1)$ and $H_{1,1} = \frac{1}{1-t}$.
\end{thm}
This proof will also be given in Section \ref{sec:Proofs of the Theorems}, after we have developed some of the consequences in the remainder of this section.

\subsection{The partially ordered set $[p] \times [q]$ }

Recall that a \emph{partially ordered set} is a set $S$ together with a relation, $\preceq$, which is required to be reflexive, antisymmetric and transitive.  In particular, given positive integers $p$ and $q$, we define $S = [p] \times [q]$, and
\[
    (i,j) \preceq (i',j') \iff i'-i \in \bbN \mbox{ and } j'-j \in \bbN
\] for $1 \leq i,i' \leq p$ and $1 \leq j,j' \leq q$,
which is a partially ordered set.

It is important to note that not all elements are comparable with respect to this order.  For example, if $p=q=2$, clearly $(1,2) \not \preceq (2,1)$ and $(2,1) \not \preceq (1,2)$.  We say that $(1,2)$ and $(2,1)$ are incomparable.  A subset of $S$ in which all pairs are comparable is called a \emph{chain}.

Given a $p$ by $q$ matrix, $J$, we define the \emph{support} as
\[
    \mbox{support}(J) := \{ (i,j) : J_{ij} > 0 \}.
\]

\begin{prop}\label{prop:poset}
Let $p$ and $q$ be positive integers, and $s \in \bbN$.
For $\mu \in \mcC(s,p)$, $\nu \in \mcC(s,q)$ and
$J \in \mcJ_{\mu \nu}$, there exists a $J' \in \mcJ_{\mu \nu}$ such that the support of $J'$ is a chain in $[p] \times [q]$, and $\langle J', C \rangle \leq \langle J, C \rangle$.
\end{prop}
\begin{proof}  Suppose there exist incomparable elements $(i,j), (i',j')$ such that $J_{ij},J_{i'j'} > 0 $.  Without loss of generality assume $i'<i$, $j<j'$, and $0< J_{ij}\leq J_{i'j'}$.  We construct $J'$ as follows:
\[
\left[
  \begin{array}{ccccc}
      &   &   &   &   \\
      &   &   & J_{i'j'} &   \\
      &   &   &   &   \\
      & J_{ij} &   &  &   \\
      &   &   &   &   \\
  \end{array}
\right]
\]

For $(k,l) \not \in \{(i,j), (i',j'), (i',j), (i,j') \}$ let $J'_{kl} = J_{kl}$.  Next, let $J'_{ij} = 0$ and
\[
    J'_{i'j'} = J_{i'j'} - J_{ij}
\]
which is non-negative.  And,
\[
    J'_{i'j} = J_{i'j} + J_{ij}
\]
\[
    J'_{ij'} = J_{ij'} + J_{ij}.
\]
The result follows.
\end{proof}
The point here is that we will only need to consider matrices $J$ with support on a chain.

\subsection{A special case of the Robinson-Schensted-Knuth correspondence}

In this subsection we recall, in detail, a special case of the Robinson-Schensted-Knuth correspondence (RSK), see \cite{F}.
For those familiar with RSK, we consider the case where the Young diagrams have only one row.  For non-experts, the exposition here does not require any knowledge of RSK.

Given $\mu \in \mcC(s,n)$ with $\mu = (\mu_1, \ldots, \mu_n)$ we define the \emph{word} of $\mu$ to be a finite weakly increasing sequence of positive integers, $w = w(\mu) = w_1 w_2 w_3 \cdots$ where the number of times $k$ occurs in $w$ is equal to $\mu_k$.

As an example, if $\mu = (3,0,2,1,0)$ then
\[
    w = 1 1 1 3 3 4.
\]
Note that the length of the word is equal to $s = \sum \mu_i$, and all components of $w$ are at most $n$.

Next, for $s\in \bbN$ and positive integers $p$ and $q$, we define
\[
  R(p,q;s) :=  \left\{ J \in \bbM_{p,q} :(\forall i,j), J_{ij} \in \bbN, \sum_{i,j} J_{ij} = s \mbox{ and support$(J)$ is a chain}   \right\}.
\]

\begin{prop}\label{prop:bijection}
For a given $s \in \bbN$ and positive integers $p$ and $q$, we have a bijection, $\Phi$, between $\mcC(s,p) \times \mcC(s,q)$ and $R(p,q;s)$.
\end{prop}
\begin{proof}
Given $(\mu,\nu) \in \mcC(s,p)\times \mcC(s,q)$, let $u,v$ be the words of $\mu$ and $\nu$ respectively,
\[
    u = u_1 u_2 u_3 \cdots u_s
\] with $1 \leq u_i \leq p$, and
\[
    v = v_1 v_2 v_3 \cdots v_s
\] with $1 \leq v_j \leq q$.
Define a $p$ by $q$ matrix by
\[
    J_{ij} = |\{ k : (u_k,v_k) = (i,j) \}|
\] for $1 \leq i \leq p$ and $1 \leq j \leq q$.  Note that the support of $J$ is a chain.
We define $\Phi(\mu,\nu) = J$.  Given $J$, we can recover $\mu$ and $\nu$ as the row and column sums of $J$.
\end{proof}

\subsection{Rank one matrices and the Segr{\'e} embedding}

We let $\bbD^{\leq k}(p,q)$ denote the set of $p$ by $q$ matrices with rank at most $k$, which is a closed affine algebraic set, called a \emph{determinantal variety}.  For a relatively recent expository article about the role these varieties play in algebraic geometry and representation theory see \cite{EHP}.

In this section we consider the $k=1$ case, in our context. Define $P : \bbR^p \times \bbR^q \rightarrow \bbM_{p,q}$ by
\[
    P(v,w) = v w^T
\] for $v \in \bbR^p$ and $w \in \bbR^q$.  Note that if $P(v,w) \neq 0$ then the rank is 1.
In fact, the image of $P$ consists of those matrices with rank at most 1.  So if $p,q > 1$ then $P$ is not surjective.

Injectivity of $P$ fails as well since for non-zero $c \in \bbR$, $v$, $w$ we have
$P(v,w) = P(cv, \frac{1}{c}w)$.   However, if we pass to projective space we recover an injective map.

In this light, let $\bbR \bbP^n$ be an $n$-dimensional real projective space, that is:
\[
    \bbR \bbP^n = \{ \bbR v : 0 \neq v \in \bbR^{n+1} \}
\] where $\bbR v$ denotes the 1-dimensional subspace spanned by non-zero $v$.  We will also write $\bbR \bbP^n := \bbP( \bbR^{n+1})$.

The \emph{Segr{\'e}} embedding,
\[
    \bbP( \bbR^p ) \times \bbP( \bbR^q ) \rightarrow \bbP( \bbR^{pq} )
\] is defined as follows: first, we note that we can identify $\bbR^{pq}$
with the $p$ by $q$ matrices by choosing bases.  Next, given an ordered pair of projective points
(i.e. one-dimensional subspaces) we can choose non-zero vectors $v$ and $w$ respectively.
The value of the Segr{\'e} embedding is the one-dimensional subspace in $\bbM_{p,q}$ spanned by the matrix $P(v,w)$.  It is easily checked that this map is well-defined and injective.

The image of the Segr{\'e} embedding gives rise to a projective variety structure on the set-cartesian product of the two projective varieties.  The projective coordinate algebra of the Segr{\'e} embedding is intimately related to $H_{p,q}(z,t)$, which we will see next.

Let $m_{ij}$ be a choice of (algebraically independent) indeterminates.  We consider the polynomial algebra
\[
    \mcA_{p,q} = \bbR[ m_{ij} : 1 \leq i \leq p, 1 \leq j \leq q].
\]

Then for $1\leq i < i' \leq p$, and $1\leq j< j' \leq q$ define
\[
    \Delta(i,i'; j,j')  =
    \left|
        \begin{array}{cc}
            m_{i  j }   & m_{i j' } \\
            m_{i' j } & m_{i' j'}
        \end{array}
    \right|.
\]
The ideal, $\mcI$, generated by the $\Delta(i,i'; j,j')$ vanishes exactly on the matrices of rank 1.  Conversely, any polynomial function that vanishes on the rank at most 1 matrices is in $\mcI$.
The algebra of coordinate functions on the rank at most 1 matrices is then isomorphic to the quotient of $\mcA_{p,q}$ by $\mcI$.  Define $\mcR(p,q) := \mcA_{p,q}/\mcI$.

The point here is that monomials involving variables which have indices that are not comparable with respect to $\preceq$ may be replaced (modulo $\mcI$) with comparable indices.  That is, $m_{ij}m_{i'j'}$ can be replaced with $m_{i'j}m_{ij'}$. This process may be thought of as ``straightening'' and is related to the non-negative integer matrices $J$ with support in a chain.  The matrix $J$ may be thought of as the exponents in a monomial.

More generally, the situation may be put into the context of Gr{\"o}bner bases.  The cost matrix $C$ used here assigns a number to each pair of indices.  This number can be used to scale the degree of $m_{ij}$.
 Using this new notion of degree, we can set up a partial order of the monomials, which can then be extended (say, lexicographically) to a well ordering of the monomials that is compatible with multiplication.  That is, we can create a term order (see \cite{CLO}).  The minors generating the ideal $\mcI$ are indeed a Gr{\"o}bner basis.  The complement of the ideal of leading terms is then a vector space basis for the quotient by $\mcI$.

For $s \in \bbN$, let $\mcA_{p,q}^s$ denote the subspace of homogeneous degree $s$ polynomials,
and set $\mcR_{p,q}^s = \mcA_{p,q}^s / (\mcA_{p,q}^s\cap \mcI)$.  Since $\mcI$ is generated by homogeneous polynomials, we have
\[
    \mcR_{p,q} = \bigoplus_{s=0}^\infty \mcR_{p,q}^s.
\]
That is, we have an algebra gradation by polynomial degree.

The polynomial functions on $\bbR^p$ (resp. $\bbR^q$) will be denoted $\mcA_p$ (resp. $\mcA_q$).  Given vectors $v \in \bbR^p$ and $w \in \bbR^q$ an element of the tensor product $\mcA_p \otimes \mcA_q$ defines a function on $\bbR^p \times \bbR^q$ with value $f(v) g(w)$.  Given an element
$(v,w) \in \bbR^p \times \bbR^q$, and $f \otimes g \in \mcA_p \otimes \mcA_q$, the value of $f \otimes g$ on $(v,w)$ is given by $f(v) g(w)$.   Extending by linearity we obtain an algebra isomorphism from the polynomials on $\bbR^p \times \bbR^q$ to the tensor product algebra $\mcA_p \otimes \mcA_q$.

The quadratic map $P$, defined above, gives rise to an algebra homomorphism,
\[
    P^* : \mcR_{p,q} \rightarrow \mcA_{p} \otimes \mcA_{q},
\] defined such that $P^*(F)$ is a function on $\bbR^p \times \bbR^q$ from a function $F$ on $\bbD^{\leq 1}_{p,q}$.  This is done via the usual adjoint map given by $[P^*(F)](v,w) = F( P(v,w))$.

The image of $P^*$ is given as the ``diagonal'' subalgebra:
\[
    \bigoplus_{s=0}^\infty \mcA_p^s \otimes \mcA_q^s.
\]

\bigskip
From our point of view, the significance of this structure is as follows:
\begin{itemize}
\item The dimension
\[
    \dim \left( \mcA_p^s \otimes \mcA_q^s \right) =
    \binom{s+p-1}{p-1} \binom{s+q-1}{q-1},
\]
which is equal to the cardinality of $\mcC(s,p) \times \mcC(s,q)$.  That is, a basis may be parameterized by a pair of compositions.
\item The (finite dimensional) vector space $\mcR_{p,q}^s$ will have a dimension also equal to the above since $P^*$ is an isomorphism.
\item The bijection $\Phi$ from Proposition \ref{prop:bijection} establishes that the above dimension is
given by the cardinality of $R(p,q;s)$.
\item A basis for $\mcR_{p,q}^s$ may be given by the monomials of the form $\prod m_{ij}^{P_{ij}}$ where $P \in R(p,q;s)$.
\item These monomials correspond to the monomials in $\mcA_p^s \otimes \mcA_q^s$ with exponents
$(\mu,\nu) \in \mcC(s,p) \times \mcC(s,q)$.
\end{itemize}

\subsection{The specialization and a derivative}

From the definition it is relatively easy to see that
\[
    H_{p,q}(0,t) = \frac{1}{(1-t)^{\min(p,q)}}.
\]

We next turn to the specialization $H_{p,q}(1,t)$, which turns out to be the Hilbert series of the rank at most 1 matrices.  That is to say
\[
    H_{p,q}(1,t) = \sum_{s = 0}^\infty (\dim \mcR_{p,q}^s) t^s.
\]
In \cite{EW} this series was computed in Equation (6.4) as
\[
    H_{p,q}(1,t) = \frac{\sum_{i=0}^{\min(p-1,q-1)} \binom{p-1}{i}\binom{q-1}{i} t^i }{(1-t)^{p+q-1}}.
\]

In this sense, $H(z,t)$ interpolates between the generating function for compositions and the Hilbert series of the determinantal varieties (at least in the rank one case).  Moreover, for generic $z$ we have a relationship to the EMD.

\bigskip
Our goal is to compute the expected value of $\cEMD$.  Therefore, it is natural to compute the partial derivative of $H_{p,q}(z,t)$ with respect to $z$, and then set $z=1$.  From the definition of $H_{p,q}$,

\[
    \left. \frac{\partial H_{p,q}(z,t)}{\partial z} \right|_{z=1} =
    \sum_{s=0}^\infty \left(\sum_{(\mu,\nu) \in \mcC(s,p) \times \mcC(s,q)} \cEMD(\mu, \nu) \right) t^s
\]

To expand this we start with
\[ H_{p,q}(z,t) = \frac{ H_{p-1,q} + H_{p,q-1} - H_{p-1,q-1} }{1- z^{|p-q|} t},  \]
the recursive relationship from Theorem \ref{thm:HS}.  Then let the partial derivative of $H_{p,q}$ with respect to $z$ be denoted $H'_{p,q}$.  We find the derivative using the ``quotient rule''

\begin{equation}
\hspace{-9cm} H'_{p,q}(z,t) =
\frac{\partial}{\partial z} H_{p,q}(z,t) =
\end{equation}
\[
\frac{(H'_{p-1,q} + H'_{p,q-1} - H'_{p-1,q-1})(1-z^{|p-q|}t) +
|p-q|z^{|p-q|-1} t (H_{p-1,q} + H_{p,q-1} - H_{p-1,q-1})} {(1-z^{|p-q|}t)^2}.
\]

When $z=1$ this becomes
\begin{equation}\label{eqn:diff z=1}
\begin{split}
H'_{p,q}(1,t) =
\frac{1}{(1-t)^2} \displaystyle \bigg( \Big(H'_{p-1,q}(1,t) + H'_{p,q-1}(1,t) - H'_{p-1,q-1}(1,t)\Big)\Big(1-t \Big)  + \\
 |p-q|t \Big(H_{p-1,q}(1,t) + H_{p,q-1}(1,t) - H_{p-1,q-1}(1,t)\Big)  \bigg).
\end{split}
\end{equation}

Before proceeding it is useful to see some initial values:

\[
\begin{array}{ccccc}
H'_{1,1}= 0 & & \displaystyle H'_{1,2}=\frac{t}{(1-t)^3} & & \displaystyle H'_{1,3}= \frac{3 t}{(1-t)^4} \\
\\ \\
\displaystyle H'_{2,1}= \frac{t}{(1-t)^3} & & \displaystyle H'_{2,2}=\frac{2 t}{(1-t)^4} & &\displaystyle H'_{2,3}=\frac{t (3 t+5)}{(1-t)^5} \\
\\ \\
\displaystyle H'_{3,1}= \frac{3 t}{(1-t)^4} & & \displaystyle H'_{3,2}= \frac{t (3 t+5)}{(1-t)^5} & & \displaystyle H'_{3,3}= \frac{8 t (t+1)}{(1-t)^6}
\end{array}
\]
\\ \\

Both $H_{p,q}(1,t)$ and $H'_{p,q}(1,t)$ are rational functions.  We anticipate that their numerators are
\[
    W_{p,q}(t) := (1-t)^{p+q-1} H_{p,q}(1,t)
\]
and
\[
    N_{p,q}(t) := (1-t)^{p+q} H'_{p,q}(1,t).
\]

Thus, multiplying by $(1-t)^{p+q}$ on both sides of Equation \ref{eqn:diff z=1} gives:
\[
N_{p,q} = \frac{1}{(1-t)^2}\Big(\big((1-t)(N_{p-1,q}+N_{p,q-1}) - (1-t)^2 N_{p-1,q-1}\big)(1-t) +
\]
\[
\hspace{4cm} |p-q|t \big((1-t)^2 (W_{p-1,q} + W_{p,q-1}) - (1-t)^3 W_{p-1,q-1}\big) \Big)
\]

or
\[
N_{p,q} =
N_{p-1,q}+N_{p,q-1} - (1-t) N_{p-1,q-1} +
|p-q| \, t \, (W_{p-1,q} + W_{p,q-1} - (1-t) W_{p-1,q-1})
\]

If we also note that
\[
    W_{p,q} = W_{p-1,q} + W_{p,q-1} - (1-t)W_{p-1,q-1},
\]
we ultimately obtain
\begin{equation}\label{eqn:N recursion}
N_{p,q} =
N_{p-1,q} + N_{p,q-1} - (1-t) N_{p-1,q-1} +
|p-q| \, t \, W_{p,q}.
\end{equation}
An easy induction shows that both $W_{p,q}(t)$ and $N_{p,q}(t)$ are polynomials in $t$.

Before proceeding it is instructive to recall our goal of finding the expected value of $\cEMD$.  In this light,
define
\[
    \mcN(p,q;s) := \sum_{(\mu,\nu) \in \mcC(s,p) \times \mcC(s,q)}
    \cEMD(\mu,\nu)
\]
for $s \in \bbN$ and positive integers $p$ and $q$.  The expected value of $\cEMD$ will be then obtained from
\[
    \lim_{s \rightarrow \infty} \displaystyle \, \, \frac{1}{s} \, \,
    \frac{\mcN(p,q;s)}{
        \binom{s+p-1}{p-1}
        \binom{s+q-1}{q-1}
    },
\]
which we will show in the proof of Theorem \ref{thm:main} to be $\mcM_{p,q}$.  In the next subsection we will use Proposition \ref{prop:N} to find the asymptotic value as $s \rightarrow \infty$ for fixed values of $p$ and $q$.  First we need more information about $\mcN(p,q;s)$.

\begin{prop}\label{prop:N}  Given positive integers $p$ and $q$,
\[
    \frac{N_{p,q}(t)}{(1-t)^{p+q}} = \sum_{s=0}^\infty \mcN(p,q;s) t^s
\]
\end{prop}
\begin{proof}
We have seen that
\[
    \left. \frac{\partial H}{\partial z} \right|_{z=1} = \frac{N_{p,q}(t)}{(1-t)^{p+q}}.
\]
Differentiating the definition $H_{p,q}(z,t)$ term by term and then setting $z=1$ gives the result.
\end{proof}

Using Equation \ref{eqn:N recursion} and the formula for $W_{p,q}(t)$ one can efficiently compute
$N_{p,q}(t)$ for specific values of $p$ and $q$.  Following the method for generating functions, we multiply $N_{p,q}(t)$ and the series expansion of $\frac{1}{(1-t)^{p+q}}$.  Because the series expansion involves only binomial coefficients we are led to an efficient method for finding the expected value of $\cEMD$ on $\mcC(s,p) \times \mcC(s,q)$ for any given values of $p$, $q$ and $s$.  Consequently we determine the values of $\mcN(p,q;s)$.

Some initial data for $N_{n,n}(t)$ for $n = 1, \cdots, 12$ are:
\[
\begin{array}{l}
0 \\
2 t  \\
8 t (t+1) \\
4 t \left(5 t^2+14 t+5\right) \\
8 t \left(5 t^3+27 t^2+27 t+5\right) \\
2 t \left(35 t^4+308 t^3+594 t^2+308 t+35\right) \\
16 t \left(7 t^5+91 t^4+286 t^3+286 t^2+91 t+7\right) \\
8 t \left(21 t^6+378 t^5+1755 t^4+2860 t^3+1755 t^2+378 t+21\right) \\
16 t \left(15 t^7+357 t^6+2295 t^5+5525 t^4+5525 t^3+2295 t^2+357 t+15\right) \\
2 t \left(165 t^8+5016 t^7+42636 t^6+142120 t^5+209950 t^4+142120 t^3+42636 t^2+5016 t+165\right) \\
8 t \left(55 t^9+2079 t^8+22572 t^7+99484 t^6+203490 t^5+ \right. \\
\hspace{7.2cm} \left. 203490 t^4+99484 t^3+22572 t^2+2079 t+55\right)\\
4 t \left(143 t^{10}+6578 t^9+88803 t^8+499928 t^7+1352078 t^6+1872108 t^5+ \right.\\
\hspace{7.675cm} \left. 1352078 t^4+499928 t^3+88803 t^2+6578 t+143 \right)
\end{array}
\]

\bigskip

The coefficients of the above are non-negative integers, which we prove inductively.  Furthermore, they apparently are \emph{palindromic} --
that is the coefficient of $t^i$ matches the coefficient of $t^{d-i}$ where $d$ is the polynomial degree.   Lastly we note that the coefficients rise in value until the middle and then decrease -- that is to say they are \emph{unimodal}.   Polynomials with these properties are often of interest.   We conjecture:

\begin{conj}
    The coefficients of the polynomials $N_{n,n}(t)$ are unimodal and palindromic.
\end{conj}

\bigskip

In Theorem \ref{thm:main}, the $\pEMD$ has been normalized so that $\mu$ and $\nu$ are probability distributions.  In terms of pairs of compositions of $s$, this amounts to multiplying by $\frac{1}{s}$.  Note that since the order of the pole at $t=1$ in $H_{p,q}(1,t)$ is one less than the order of the pole at $t=1$ in $H'(1,t)$ we see that the (normalized) EMD is approaching a constant as $s \rightarrow \infty$.  Alternatively, we could choose not to normalize and then obtain a linear growth of $s \, \mcM_{p,q}$.

The following is a table with approximate values of $\mcM_{p,q}$ for $1 \leq p,q, \leq 5$,
\[
\begin{array}{c|ccccc}
p \backslash q & 1     & 2     & 3     & 4     & 5      \\ \hline
1 & 0.000 & 0.500 & 1.000 & 1.500 & 2.000  \\
2 & 0.500 & 0.333 & 0.667 & 1.100 & 1.567  \\
3 & 1.000 & 0.667 & 0.533 & 0.800 & 1.190  \\
4 & 1.500 & 1.100 & 0.800 & 0.686 & 0.914  \\
5 & 2.000 & 1.567 & 1.190 & 0.914 & 0.813
\end{array}
\]
We display a 3D plot for $1\leq p,q \leq 12$ in Figure \ref{fig:3D}.
\begin{figure}
\includegraphics{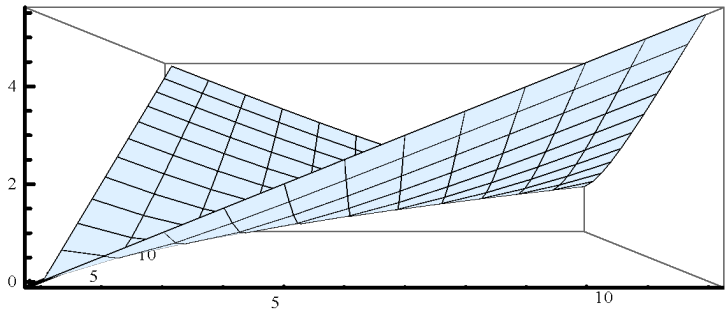}
\caption{3D plot of $\mcM_{p,q}$ for $1 \leq p,q \leq 12$}
\label{fig:3D}
\end{figure}

\section{Further Calculations and Proofs of the theorems}\label{sec:Proofs of the Theorems}

This section includes the main technical points of the paper, including the proofs of the main theorems.
It is more convenient to prove Theorem \ref{thm:main} after Theorems \ref{thm:L1} and \ref{thm:HS}.

First, however, we make explicit two useful scaled versions of the EMD.

\subsection{Unit normalized Earth Mover's Distance} $\,$

Averaging $\frac{\mcE(\mu-\nu)}{s}$ over $(\mu,\nu) \in \mcC(s,n) \times \mcC(s,n)$ gives rise to the expected value of the discrete EMD.  Taking the limit as $s \rightarrow \infty$ gives the expected value of the \emph{normalized EMD} on
$\mcP_n$.
Observe that the maximum value of the normalized EMD on $\mcP_n$ is $n-1$. The \emph{unit normalized EMD} will be defined as the normalized EMD scaled by $\frac{1}{n-1}$.  This scaling makes $\mcP_n$ into a metric space with diameter 1. When
working with real data in the last section of this paper, we will always use the unit normalized
Earth Mover's Distance.

As an example, we consider the discrete case where $\mcP_n$ is replaced by $\mcC(s,n)$.  In particular, choose $s=30$ and $n=5$ and calculate the exact histogram for the unit normalized distance. The mean of the distribution is obtained by expanding
\[ \frac{8 t \left(5 t^3+27 t^2+27 t+5\right)}{(1-t)^{10}} \]
as a series around $t=0$, then taking the coefficient of $t^{30}$ and dividing by
$\binom{30+5-1}{5-1}^2$ (the number of ordered pairs of distributions).
The approximate value is 26.2938.

For the unit normalized distance we divide by $s(n-1) = 30(5-1) = 120$.  That is, we divide by $s$ to obtain a probability distribution, and then divide by $n-1$ to scale the diameter to 1.
The unit normalized mean is approximately 0.219115 as shown in Figure \ref{fig:hist_30n5}.

\begin{figure}
\includegraphics{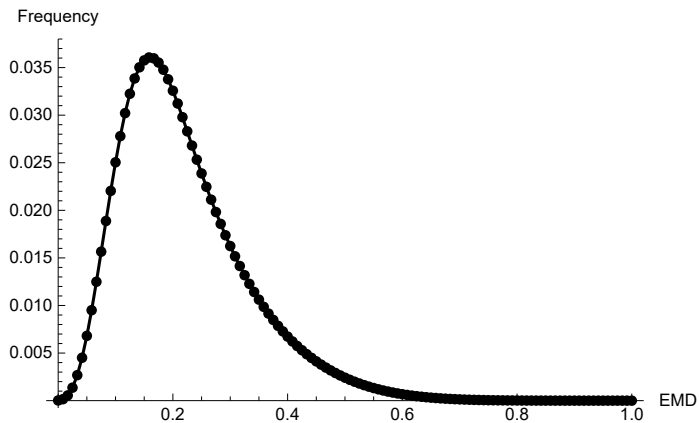}
\caption{Histogram of the unit normalized EMD for $s=30$, $n=5$}
\label{fig:hist_30n5}
\end{figure}

In the limiting case as $s \rightarrow \infty$, the mean decreases slightly from the $s=30$ case.  Again, the scaling sets the diameter of the metric space $\mcP_n$ to 1.  Thus, the expected value of the unit normalized EMD is
\[
    \widetilde \mcM_n :=     \frac{\mcM_n}{n-1}.
\]
Recall the notation that $\mcM_n = \mcM_{n,n}$, for positive integer $n$.

It should be noted that all values of $\widetilde \mcM_n$ are rational numbers.  We present the following approximate values
of $\widetilde \mcM_n$ for $n = 2, \cdots, 12$,
{\small
\[
\begin{array}{c||ccccccccccc}
n & 2 & 3 & 4 & 5 & 6 & 7 & 8 & 9 & 10 & 11 & 12 \\ \hline \\[-4mm]
 \widetilde \mcM_n
  & 0.3333
  & 0.2667
  & 0.2286
  & 0.2032
  & 0.1847
  & 0.1705
  & 0.1591
  & 0.1498
  & 0.1419
  & 0.1351
  & 0.1293
\end{array}
\]}
which are the limiting values as $s \rightarrow \infty$.  For finite choices of $s$ and $n$ we can compute the exact mean of the unit normalized EMD by expanding
\[
    \frac{N_{p,q}(t)}{(1-t)^{p+q}}
\] in the case when $p=q=n$, and then dividing by $s(n-1) \binom{s+n-1}{n-1}^2$.  We show the approximate values in the table below.
\[
\begin{array}{c||cccccc}
s \backslash n & 2 & 3 & 4 & 5 & \cdots & 12 \\ \hline
 1 & 0.5000 & 0.4444 & 0.4167 & 0.4000 & & 0.3611 \\
 2 & 0.4444 & 0.3889 & 0.3600 & 0.3422 & & 0.2991 \\
 3 & 0.4167 & 0.3600 & 0.3300 & 0.3113 & & 0.2649 \\
 4 & 0.4000 & 0.3422 & 0.3113 & 0.2918 & & 0.2428 \\
 5 & 0.3889 & 0.3302 & 0.2985 & 0.2784 & & 0.2272 \\
 10 & 0.3636 & 0.3020 & 0.2681 & 0.2462 & & 0.1881 \\
 15 & 0.3542 & 0.2912 & 0.2561 & 0.2333 & & 0.1716 \\
 20 & 0.3492 & 0.2854 & 0.2497 & 0.2264 & & 0.1624 \\
 30 & 0.3441 & 0.2794 & 0.2430 & 0.2191 & & 0.1524 \\
 60 & 0.3388 & 0.2732 & 0.2360 & 0.2114 & & 0.1415 \\
 120 & 0.3361 & 0.2700 & 0.2323 & 0.2073 & & 0.1355 \\
 180 & 0.3352 & 0.2689 & 0.2311 & 0.2060 & & 0.1335 \\
 360 & 0.3343 & 0.2678 & 0.2298 & 0.2046 & & 0.1314 \\
 500 & 0.3340 & 0.2675 & 0.2295 & 0.2042 & & 0.1308 \\
 750 & 0.3338 & 0.2672 & 0.2292 & 0.2039 & & 0.1303 \\
 1000 & 0.3337 & 0.2671 & 0.2290 & 0.2037 & & 0.1300 \\
 1250 & 0.3336 & 0.2670 & 0.2289 & 0.2036 & & 0.1299 \\
 1500 & 0.3336 & 0.2669 & 0.2289 & 0.2035 & & 0.1298 \\
 2000 & 0.3335 & 0.2669 & 0.2288 & 0.2034 & & 0.1296 \\
 10000 & 0.3334 & 0.2667 & 0.2286 & 0.2032 & & 0.1293 \\
\end{array}
\]

\subsection{Proof of Theorem \ref{thm:L1}}

\begin{proof}
By continuity it will suffice to show that this is true on a dense subset of $\mcP_n$.  Specifically, we will consider the special case that
if $\mu$ (resp. $\nu$) is of the form
\[
    \mu = \left( \frac{a_1}{s}, \cdots, \frac{a_n}{s} \right)
\] for some positive integer $s$ and $(a_1, \cdots, a_n) \in \mcC(s,n)$.
As $s \rightarrow \infty$, such points are dense in $\mcP_n$

For $p \leq n$ (resp. $q \leq n$) we can regard $\mcC(s,p)$
(resp. $\mcC(s,q)$) as being embedded in $\mcC(s,n)$ by appending zeros onto the right.

By induction on $p+q$ we will show that for any non-negative integer $s$,
\[
    \cEMD(\mu,\nu) = \mcE(\mu-\nu)
\] for $\mu \in \mcC(s,p)$ and $\nu \in \mcC(s,q)$.

If $p=q=1$ the result is trivial since there is only one composition of $s$.
Consider $p+q \geq 2$.
Without loss of generality assume $p \leq q$.

We proceed by induction on $s$ (inside the induction on $p+q$).
If $s = 0$ the statement is vacuous, and so the base case is clear.

For positive integer $s$ let $J$ be a p-by-q non-negative integer matrix such that $J$
has row and column sums $\mu$ and $\nu$ respectively and
$\langle J, C \rangle$ is minimal.  We shall show that
$\langle J, C \rangle = \mcE(\mu-\nu)$.

If the first row (resp. column) of $J$ is zero we can delete it and reduce to the inductive
hypothesis on $p+q$.  Therefore, we assume that there is a positive entry in the first row (resp. column) of $J$.
If $J_{11} > 0$ then we can subtract $J_{11}$ from $s$
and reduce to the inductive hypothesis on $s$.

We are therefore left with $J_{11} = 0$ and
the existence of $i>1$, $j>1$ with $J_{1j}>0$ and $J_{i1}>0$.
However, $(1,j)$ and $(i,1)$ are incomparable in the poset
$[p] \times [q]$.  However, by Proposition \ref{prop:poset} we can assume that the
support of $J$ is a chain.

The cost for the joint distribution, $J$, is minimized when the support is a chain.   Upon inspection, one sees that the value of $\mcE$ agrees with the cost in the case that the support of $J$ is on a chain.  The value of $J_{i,j}$ is multiply counted $|i-j|$ times -- once for each of the contributing terms of $\mcE$.  In the non-contributing terms there is a telescoping  inside the absolute value.

Theorem \ref{thm:L1} follows for chains, to which we have inductively reduced the problem.
\end{proof}

\subsection{Proof of Theorem \ref{thm:HS}}

\begin{proof}
The vector space of degree $s$ homogeneous polynomial functions on the rank at most one p-by-q matrices is denoted $\mcR_{p,q}^s$.  By Proposition \ref{prop:poset}, we obtain a basis for this space by considering the monomials
\[
    \prod_{i=1}^p \, \, \prod_{j=1}^q x_{ij}^{J_{ij}}
\]
where $J$ is a non-negative integer matrix with support on a chain.  The row and column sums of $J$ are a pair of compositions of $s$ with $p$ and $q$ parts respectively.  We denote these by $\mu$ and $\nu$.   Note that by Proposition \ref{prop:bijection}, $\mu$ and $\nu$ determine J.

If we assign each of these monomials the formal expression $z^{\cEMD(\mu,\nu)} t^s$ and sum them as formal series, we obtain the Hilbert series of $\mcR_{p,q}$,
\[
    \sum_{s=0}^\infty \left( \sum_{(u,v) \in \mcC(s,p) \times \mcC(s,q)}
                z^{\tiny \cEMD(u,v)} \right) t^s
\]
which we then recognize as the definition of $H_{p,q}(z,t)$.

Each monomial has a non-negative integer matrix $J$ as its exponents, with support on a chain.  This chain terminates at or before $x_{p,q}^{J_{p,q}}N$.  From Theorem \ref{thm:L1} the variable $x_{p,q}$ is multiplied by $z^{|p-q|} \, t$, and contributes
\[
    \sum_{J_{p,q} = 0}^\infty \left( z^{|p-q|} t \right)^{J_{p,q}}
\]
to all monomials.  The geometric series sums to $\displaystyle \frac{1}{1-z^{|p-q|} \, t}$.

The preceding variables in the monomial may contain $x_{p,j}$ for some $1 \leq j \leq q$, or $x_{i,q}$ for some $1\leq i \leq p$, but not both -- since the exponent matrix has support in a chain.  In the former case these monomials are in the sum
$H_{p,q-1}$, while in the latter are counted in $H_{p,q-1}$.

The sum $H_{p-1,q} + H_{p,q-1}$ over counts monomials.  That is to say, if a monomial has an exponent
with support involving variables $x_{i,j}$ with $i<p$ and $j<q$ then it is counted once in $H_{p-1,q}$ and once in $H_{p,q-1}$.  It also appears once in $H_{p-1,q-1}$.  We therefore observe that such monomials are counted exactly once in the expression
\[
    H_{p-1,q} + H_{p,q-1} - H_{p-1, q-1}.
\]

Finally, we see that all such monomials are counted exactly once in the product
\[ \frac{ H_{p-1,q} + H_{p,q-1} - H_{p-1,q-1} }{1- z^{|p-q|} t}  \]
if $(p,q) \neq (1,1)$ and $H_{1,1} = \frac{1}{1-t}$.
\end{proof}

\subsection{Proof of Theorem \ref{thm:main}}
\begin{proof}

Fix positive integers $p$ and $q$.  The coefficient of $t^s$ in
$\frac{1}{(1-t)^{p+q}}$ is
\[
    \binom{s+p+q-1}{p+q-1} = \frac{s^{p+q-1}}{(p+q-1)!} + \mbox{ lower order terms in }s.
\]

And, we have
\[
    N_{p,q}(t) = c_0 + c_1 t + c_2 t^2 + \cdots + c_k t^k
\] for some non-negative integers $c_0, \cdots, c_k$.
Thus the coefficient of $t^s$ in
\[
    \frac{N_{p,q}(t)}{(1-t)^{p+q}}
\]
is therefore asymptotic to
\[
    N_{p,q}(1) \frac{s^{p+q-1}}{(p+q-1)!}
\]

We will next find an inhomogeneous three term recursive formula for $N_{p,q}(1)$
from Equation \ref{eqn:N recursion}.

First we observe that $W_{p,q}(1) = \binom{p+q-2}{p-1}$.  Therefore, we have
\[
    N_{p,q}(1) = N_{p-1,q} + N_{p,q-1} + |p-q| \binom{p+q-2}{p-1}.
\]

Then we divide by $(p+q-1)!$ to obtain the asymptotic.  However, our goal is to obtain the expected value of $\cEMD$.
So, in light of Proposition \ref{prop:N}, we will need to divide by
\[
    \binom{s+p-1}{p-1} \binom{s+q-1}{q-1} \sim   \frac{s^{p+q-2}}{(p-1)!(q-1)!}.
\]

Thus, we find that the expected value is
\[
    \mcM_{p,q} = \frac{(p-1)! (q-1)!}{(p+q-1)!} N_{p,q}(1),
\]
which we can rewrite as
\[
    \mcM_{p,q} = \frac{(p-1)\mcM_{p-1,q} + (q-1)\mcM_{p,q-1} + |p-q|}{p+q-1}.
\]
\end{proof}

\section{Relation to spectral graph theory}\label{sec:Spectral}

We next turn our attention to some results from spectral graph theory and describe a connection with the expected value of the EMD.

The concept of a graph (or network) is likely familiar to the reader.   We recall the terminology briefly.
By a \emph{graph} we mean an ordered pair, $(V,E)$, where $V$ is a finite set whose elements are called \emph{vertices} and $E$ is a finite set whose elements are called \emph{edges}, together with an injective mapping from $E$ to unordered pairs of distinct vertices.  The elements of $E$ are said to \emph{join} the corresponding pair of vertices.  The number of vertices joined to a given vertex, $v \in V$, is called the degree of $v$, denoted $deg(v)$.

A sequence of distinct vertices, $v_1, v_2, \cdots, v_t$ with $v_i$ joined to $v_{i+1}$ for each $i$ is called a \emph{path}.   If a path exists between between all pairs of vertices then we say that the graph is \emph{connected}.

Given vertices $v$ and $w$ in a connected graph, the \emph{distance} between vertex $v$ and vertex $w$ is the length of the shortest path starting with $v$ and ending with $w$, and will be denoted $\rho(v,w)$.  The function $\rho$ is a metric on $V$.
For an integer $r$, the \emph{ball} of radius $r$ centered at $v \in V$ will be defined as
\[
    B_r(v) := \{ w \in V : \rho(v,w) \leq r \}.
\]
Furthermore, let $n_r(v) = |B_r(v)| - |B_{r-1}(v)|$, and set
\[
    S(v) := \sum_{r \geq 0} r n_r(v),
\] which is a finite sum giving the expected distance a vertex is from $v$.

As we shall see, this metric is related to the topic of this article.
Specifically, the \emph{mean distance} in a graph is defined to be:
\[
    \overline \rho(G) := \frac{1}{m(m-1)} \sum_{v \in V} S(v).
\] where $m = |V|$.

The above notation is from \cite{M2}, Section 3.  Observe that the definition is equivalent to averaging the the distance between all two-element subsets
of $V$.  We also point out that if we consider all ordered pairs of vertices we have:
\[
    \left( 1-\frac{1}{m} \right) \overline \rho(G) = \frac{1}{m^2} \sum_{(v,w) \in V \times V} \rho(v,w).
\]

From our point of view, we consider the graph to be on the vertex set $\mcC(s,n)$ where two
compositions are joined when the (unnormalized) Earth Mover's Distance is exactly 1.  We call this graph the \emph{Earth Mover's Graph}, denoted $G(s,n)$.  In this case the length of the shortest path between two vertices is the Earth Mover's Distance.

The average distance between ordered pairs of vertices in $G(s,n)$ is the subject of this article.  More generally, one can consider a graph $G$ whose vertices are a subset of $\mcC(s,n)$ and two vertices $v$ and $w$ are joined when $\cEMD(v,w) \leq t$
for some fixed constant $t \geq 0$.   We will call $t$ the \emph{threshold}.
In practice, determining values of the threshold that uncover features in the data is an
important research topic.   Understanding the expected EMD in the uniform case is only one line of research.

The connected components of $G$ are of interest in cluster analysis.  For example, the connected components of $G$ may be interpreted as clusters.
If $t=0$ there are no edges in $G$, thus no connected components.  For sufficiently large $t$, every pair of vertices is joined and there is only one component.  As $t$ decreases the graph disconnects.   Hierarchical connection of components defined by $t$ gives rise to many different types of clustering.  An example of this type of analysis will be given in the next section.

There are several ``off the shelf'' methods for clustering analysis.  See \cite{DM} Chapter 20 for some commentary, especially about the popular ``k-means'' algorithm.  Here we present only \emph{metric hierarchical clustering} and \emph{spectral clustering} as they relate to Theorem \ref{thm:main}.  However, the data that we consider in this article can and should be looked at from several points of view.  In particular, unsupervised machine learning techniques are merited.  See \cite{HTF} as a general reference.

The average distance of a graph is also related to other invariants; we recommend the survey \cite{M1}.  First we recall some additional terminology.  Given a graph $G$ with vertices $\{ v_1, \cdots, v_m\}$ and $k$ edges, one can form the Laplacian matrix $L_G = D_G - A_G$ where $D_G$ is the diagonal matrix with the degree of vertex $v_i$ in the i-th row and i-th column, while $A(G)$ is the adjacency matrix in which the entry in row i and column j is a one if $v_i$ and $v_j$ are joined by an edge, and zero otherwise.

The spectrum of $L_G$ is of interest.  To begin, $L_G$ is a positive semidefinite matrix.  The multiplicity of the 0-eigenspace is equal to the number of connected components of $G$.
If the spectrum of $L_G$ is denoted by $0 = \la_1 \leq \la_2 \leq \cdots \leq \la_m$,
the \emph{algebraic connectivity} is given by $\la_2$.  Intuitively, we expect ``clustering'' when $\la_2$ is small relative to the rest of the spectrum.

Related to the algebraic connectivity are inequalities proved in \cite{M2}.   We recall them here because they partially describe the structure of the graph based on $\la_2$.  In fact we can use $\la_2$ to calculate bounds on the average distance between vertices in a graph and another invariant to be defined next.

The discrete Cheeger inequality asserts
that the \emph{isoperimetric number}, $i(G)$, is closely related to the spectrum of a graph.  This number is defined as
\[
    i(G) := \min \left\{ \frac{|\delta X|}{|X|} : X \subseteq V(G) \mbox{s.t.} 0 < |X| < \frac{1}{2}|V(G)| \right\}
\] where $\delta(X)$ is defined to be the boundary of a set of vertices $X$ (that is $v \in X$ iff $v$ is in $X$ but is joined to a vertex not in $X$).
One result from \cite{M2} is
\begin{equation}\label{eq:isop}
    \frac{\la_2}{2} \leq i(G) \leq \sqrt{\la_2(2 d_{max}  - \la_2)}
\end{equation}
where $d_{max}$ is the maximum degree of a vertex in $G$.  These results have their underpinnings in geometry and topology, see \cite{Mgw}, for example.  Intuitively, the point here is that if $G$ has two large subgraphs that are joined only by a small set of edges then $i(G)$ is small.  Unfortunately, computing $i(G)$ exactly is difficult.  However, the spectrum of $G$ can be computed more easily, providing the stated bounds for $i(G)$.

A third invariant for $G$ is $\overline \rho(G)$, the mean distance.  This value can also be bounded.  The inequality presented in \cite{M2} is
\begin{equation}\label{eq:inequality}
    \frac{1}{m-1}\left(\frac{2}{\la_2} +\frac{m-2}{2}\right)
    \leq \overline \rho(G) \leq
    \frac{m}{m-1}\left[ \frac{d_{max} - \la_2}{4 \la_2} \ln(m-1) \right]
\end{equation}
where $m$ is the number of vertices in $G$.

\section{Real world data}\label{sec:Real world data}

In this section we consider a real world data set coming from the Section Attrition and Grade Report published by the Office of Assessment and Institutional Research at the University of Wisconsin - Milwaukee for Fall semesters of academic years 2013-2017.  We selected data for courses with enrollments greater than 1,000.  Analysis is done for the 12 grades A through F (with plus/minus grading). ``W" grades were not reported by the University for the entire period and are not included.

Most universities collect and analyze similar types of retention and attrition data.  Analysis of grade distribution data may be an as-yet untapped portal for self-examination that reveals patterns in large, multi-section courses, or allows for insight into cross-divisional course boundaries.  Cluster analysis provides a visual reinforcement of the calculated EMD.

The data for 2013-2017 comprises a total of twenty one courses.  In Fall 2013 there are five courses: English 101 and 102, Math 095 and 105, and Psychology 101.  Math 095 was redesigned after Fall 2013 and enrollment dropped below 1,000; the other four courses were offered every year.  Input data are sorted by division and year.

\[
\begin{array}{c||cccc}
 ID & Division & Course & Year & Enrollment \\ \hline
1 & \text{English} & \text{101} & \text{2013} & 1800 \\
 2 & \text{English} & \text{102} & \text{2013}  &1224\\
 3 & \text{English} & \text{101} & \text{2014}  &1762\\
 4 & \text{English} & \text{102} & \text{2014} &1299 \\
 5 & \text{English} & \text{101} & \text{2015}  &1742 \\
 6 & \text{English} & \text{102} & \text{2015} &1525 \\
 7 & \text{English} & \text{101} & \text{2016} &1693\\
 8 & \text{English} & \text{102} & \text{2016}  &1410 \\
 9 & \text{English} & \text{101} & \text{2017}  &1569\\
 10 & \text{English} & \text{102} & \text{2017} &1142 \\
 11 & \text{Math} & \text{95} & \text{2013} &1166  \\
 12 & \text{Math} & \text{105} & \text{2013}  &1555 \\
 13 & \text{Math} & \text{105} & \text{2014}   &1701\\
 14 & \text{Math} & \text{105} & \text{2015}  &1466\\
 15 & \text{Math} & \text{105} & \text{2016}  &1604\\
 16 & \text{Math} & \text{105} & \text{2017}  &1732\\
 17 & \text{Psychology} & \text{101} & \text{2013} &1507  \\
 18 & \text{Psychology} & \text{101} & \text{2014}&1443\\
 19 & \text{Psychology} & \text{101} & \text{2015} &1337 \\
 20 & \text{Psychology} & \text{101} & \text{2016} &1192  \\
 21 & \text{Psychology} & \text{101} & \text{2017}  &1333 \\
\end{array}
\]

\subsection{Histogram of EMD sample}
We form the Earth Mover's Graph by computing the unit normalized EMD for each pair of the 21 courses.  A histogram of the results is presented in Figure \ref{fig:hist_uwm}.  The rough structure reflects some aspects of the distribution of the theoretical case for 30 students and 5 grades depicted in Figure \ref{fig:hist_30n5}.  For example, it is almost unimodal and skewed to the right.  It is also interesting to note that the histogram in Figure \ref{fig:hist_uwm} shows maximal EMD between courses at around 0.25.  Additionally, for each course one can compute the average EMD to all others.  This course pair-wise average
has a minimum EMD of 0.067, mean 0.086, and maximum 0.129.

\begin{figure}
\includegraphics{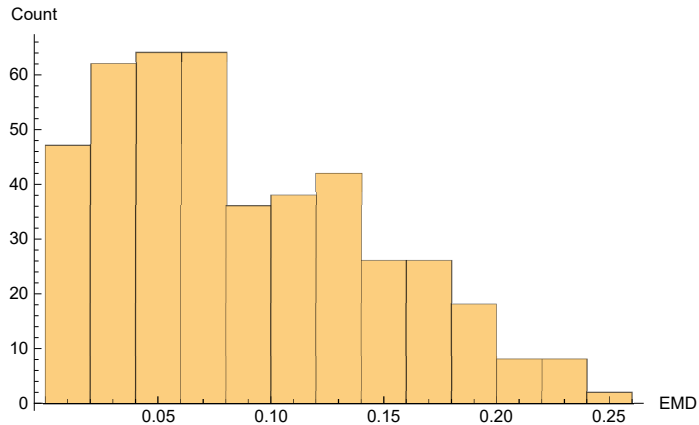}
\caption{Histogram of EMD for UWM Grade Data}
\label{fig:hist_uwm}
\end{figure}

To better understand the relative sizes of these numbers, we can compare them with the mean in the uniform model with $n=12$, which is approximately 0.1300 when $s=1000$ and then drops to 0.1293 when $s \rightarrow \infty$.  For actual grade distribution data, not all grade distributions are equally likely.  Nonetheless, for this data set the maximum mean EMD corresponds to Psychology 101, Fall 2014 and is surprisingly close to the theoretical value.  No doubt this is an anomaly, but we were surprised by how large the unit normalized EMD was between courses when compared to uniform sampling.

\subsection{Hierarchical Clustering}
Using the unit normalized EMD, we next determine a threshold, $t$, such that two courses are joined by an edge when the EMD falls below $t$.  Thus, we obtain a family of graphs parameterized by $t$.   We consider the connected component structure for each value of $t$.  In practice, one sees a
single giant component when the threshold is ``large''.   As a heuristic, we consider large to mean
greater than the expected distance in the uniform model.

\begin{figure}
\includegraphics{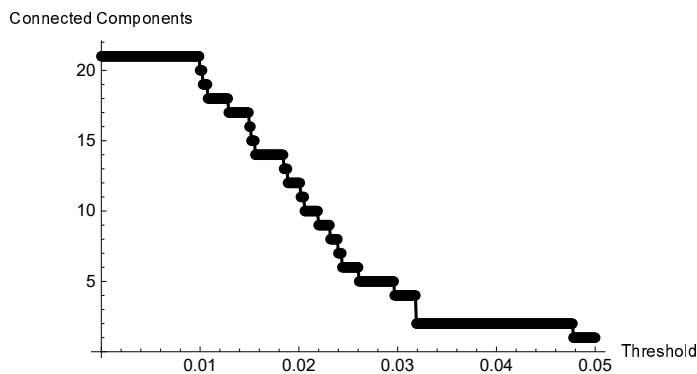}
\caption{Number of Connected Components with change in Threshold}
\label{fig:con_comp}
\end{figure}

In Figure \ref{fig:con_comp} we step through various distance threshold values $t$ and count the number of connected components. When $t$ is very small, the graph has 21 components.  As $t$ increases, the components of the graph continue to connect, with the largest range of persistence for $t \in [0.0319,0.0477]$ with two components.
It is useful to compare these values to the expected EMD in the uniform model.
The endpoints of this interval are 25\% and 37\% of the uniform mean, respectively.

Set $t=0.0478$. We form a graph on the vertex set of these 21 courses.   Two courses are joined by an edge when the unit normalized distance between them falls below $t$.  Since $t$ is larger than 0.0477, there is only one connected component as depicted in Figure \ref{fig:CC}.

\begin{figure}
\includegraphics{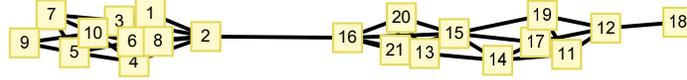}
\caption{Connected Components}
\label{fig:CC}
\end{figure}

\subsection{Spectral Analysis}
A key component of spectral clustering analysis is the construction of the Laplacian Matrix $L_G$ as defined in Section \ref{sec:Spectral}.  The second smallest eigenvalue of $L_G$ gives the algebraic connectivity.  Figure \ref{fig:spectrum} presents a plot of the full spectrum for $t =0.0478$.
\begin{figure}
\includegraphics{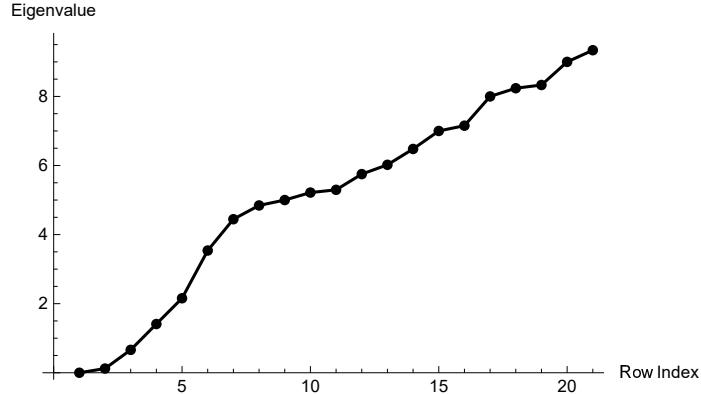}
\caption{Eigenvalues for $t=0.0478$}
\label{fig:spectrum}
\end{figure}

For the 21 vertex graph in Figure \ref{fig:CC}, the algebraic connectivity is $\la_2 \cong 0.1213$, and the maximum degree is $d_{max} = 8$.  It appears as if elements $2$ and $16$ create a tenuous bridge between two ``clusters''.

The theoretical bounds on the isoperimetric number from Equation \ref{eq:isop} are
\[
    \frac{0.1213}{2} \leq i(G) \leq \sqrt{0.1213(2 \times 8  - 0.1213)}
\]
\[
    0.06065 \leq i(G) \leq 1.3878
\]
with a computed value of $i(G) \cong 0.1$ for this data set.  The relatively small isoperimetric number is consistent with the single edge of connection between elements $2$ and $16$ in Figure \ref{fig:CC}.

The theoretical bounds on the mean distance $\overline \rho(G)$ from Equation \ref{eq:inequality} are
\[
    \frac{1}{21-1}\left(\frac{2}{0.1213} +\frac{21-2}{2}\right)
    \leq \overline \rho(G) \leq
	    \frac{21}{21-1}\left[ \frac{ 8 - 0.1213}{4 \times 0.1213} \ln(21-1) \right]
\]
\[
    1.2995  \leq \overline \rho(G) \leq 51.08
\]
and one can compute $\overline \rho(G) \cong 2.910$ for this data set.

The most curious part of this analysis appears in the plot of the components.
At threshold $t = 0.0477$ the algebraic connectivity (i.e. $\la_2$) goes to zero,
and the EMD separates all of the English courses from the cluster of Math and Psychology courses.  More specifically,
EMD splits English (Course ID's $1$ through $10$) off from Math (ID's $11$ - $16$) and Psychology (ID's $17$-$21$).  Or equivalently, grade distributions in English courses are most similar to grade distributions in other English courses, and least similar to grade distributions in both Math and Psychology courses.

\begin{figure}
\includegraphics{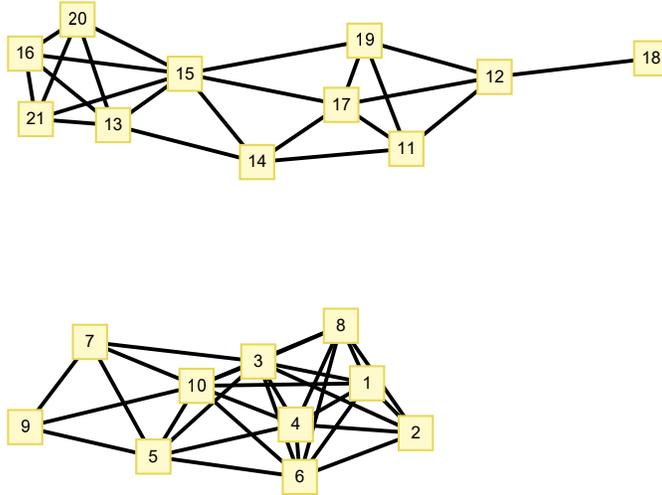}
\caption{Partitioned Components}
\label{fig:PC}
\end{figure}

Although we see this clear partition of the data in Figure \ref{fig:PC}, one needs to be cautious about clustering algorithms.
For example, clustering identified in one algorithm may not be the same as another.
See the paper \cite{K} for a careful treatment of this topic in general.   In our more specific setting the clustering is determined by the threshold $t$.  Thus, the question of where to set this value is delicate:  too small and we see too many components, while too large we see very little clustering at all.  Furthermore, our intention is only for exploratory data analysis and not to
suggest policy regarding instructional assessment.

\end{document}